\documentclass[11pt,a4paper,reqno]{amsart}
\usepackage{graphics}
\usepackage{epsfig}
\usepackage{yfonts}
\usepackage{graphicx}
\usepackage[matrix,arrow,tips,curve,ps]{xy}
\usepackage{enumerate, amssymb}
\usepackage{hyperref}
\usepackage{comment}

\def\bdi{\begin{diagram}}
\def\edi{\end{diagram}}

\newtheorem{thm}{Theorem}[section]
\newtheorem{cor}[thm]{Corollary}
\newtheorem{lem}[thm]{Lemma}
\newtheorem{prop}[thm]{Proposition}

\theoremstyle{definition}
\newtheorem{defi}[thm]{Definition}
\newtheorem{defis}[thm]{Definitions}
\newtheorem{conj}[thm]{Conjecture}
\newtheorem{quest}[thm]{Question}
\newtheorem{conv}[thm]{Convention}
\newtheorem{nota}[thm]{Notation}
\newtheorem{rem}[thm]{Remark}
\newtheorem{rems}[thm]{Remarks}
\newtheorem{exa}[thm]{Example}
\newtheorem{exas}[thm]{Examples}
\newcommand{\rien}[1]{}

\newcommand{\br}{ \operatorname{{\rm br}}}

\newcommand{\Sing}{ \operatorname{{\rm Sing}}}
\newcommand{\Seg}{ \operatorname{{\rm Seg}}}

\newcommand{\HH}{\mathcal{H}}

\newcommand{\C}{\ensuremath{\mathbb{C}}}
\newcommand{\D}{\ensuremath{\mathbb{D}}}

\newcommand{\Q}{\ensuremath{\mathbb{Q}}}

\newcommand{\cO}{{\ensuremath{\mathcal{O}}}}
\newcommand{\cC}{{\ensuremath{\mathcal{C}}}}
\newcommand{\cH}{{\ensuremath{\mathcal{H}}}}

\newcommand{\Gr}{\rm Gr}

\def\PP{{\mathbb P}}

\def\deg{\mathop{\rm deg}}

\renewcommand{\epsilon}{\varepsilon}
\renewcommand{\phi}{\varphi}

\newcommand{\bnum}{\begin{enumerate}}
\newcommand{\enum}{\end{enumerate}}
\renewcommand{\emptyset}{\varnothing}
\newcommand{\brem}{\begin{rem}}
\newcommand{\brems}{\begin{rems}}
\newcommand{\erem}{\end{rem}}
\newcommand{\erems}{\end{rems}}
\newcommand{\bexa}{\begin{exa}}
\newcommand{\bexas}{\begin{exas}}
\newcommand{\eexa}{\end{exa}}
\newcommand{\eexas}{\end{exas}}
\newcommand{\bdefi}{\begin{defi}}
\newcommand{\edefi}{\end{defi}}
\newcommand{\bdefis}{\begin{defis}}
\newcommand{\edefis}{\end{defis}}
\newcommand{\bcor}{\begin{cor}}
\newcommand{\ecor}{\end{cor}}
\newcommand{\blem}{\begin{lem}}
\newcommand{\elem}{\end{lem}}
\newcommand{\bconv}{\begin{conv}}
\newcommand{\econv}{\end{conv}}
\newcommand{\bconj}{\begin{conj}}
\newcommand{\econj}{\end{conj}}
\newcommand{\bprop}{\begin{prop}}
\newcommand{\eprop}{\end{prop}}
\newcommand{\bthm}{\begin{thm}}
\newcommand{\ethm}{\end{thm}}
\newcommand{\bnota}{\begin{nota}}
\newcommand{\enota}{\end{nota}}
\newcommand{\bsit}{\begin{sit}}
\newcommand{\esit}{\end{sit}}
\newcommand{\be}{\begin{eqnarray}}
\newcommand{\ee}{\end{eqnarray}}
\newcommand{\bproof}{\begin{proof}}
\newcommand{\eproof}{\end{proof}}
\def\ba{\begin{array}}
\def\ea{\end{array}}

\title[Scrolls and hyperbolicity]{
Scrolls and hyperbolicity}

\author{C.\ Ciliberto,
M.\ Zaidenberg}
\address{Dipartimento di Matematica, Universit\`a degli
Studi di Roma ``Tor Vergata'', Via della Ricerca Scientifica, 00133
Roma, Italy, phone: +39-06-7259-4684, fax: +39-06-7259-4699}
\email{cilibert@mat.uniroma2.it}
\address{Universit\'e Grenoble I, Institut Fourier, UMR 5582
CNRS-UJF, BP 74, 38402 Saint Martin d'H\`eres c\'edex, France,
phone: +33-476-51-4324, fax: +33-476-51-4478}
\email{Mikhail.Zaidenberg@ujf-grenoble.fr}

\thanks{
{\bf Acknowledgements:} This research was done during a visit of
the first author at the Institut Fourier, Grenoble, and of the
second author at the Universit\`a degli Studi di Roma ``Tor
Vergata''.  The second author also profited from a support from the
cooperation program GRIFGA (Groupement de Recherche europ\'een
Italo-Fran\c{c}ais en G\'eom\'etrie Alg\'ebrique). The authors thank
these institutions and the program for the generous support and
excellent working conditions.}

\thanks{
\mbox{\hspace{11pt}}{\it 2010 Mathematics Subject Classification}:
14N25, 14J70, 32J25, 32Q45.\\
\mbox{\hspace{11pt}}{\it Key words}: Kobayashi hyperbolicity,
algebraic hyperbolicity, projective hypersurface, geometric genus}

\date{}
\begin{document}

\begin{abstract} Using degeneration to scrolls, we give an easy proof
of non--existence of curves of low genera on general surfaces  in
$\PP^ 3$ of degree $d\ge 5$. We show, along the same lines,
boundedness of families of
curves of small enough genera on general surfaces in $\PP^ 3$.
We also show that there
exist Kobayashi hyperbolic surfaces in $\PP^ 3$ of degree $d=7$ (a
result so far unknown), and give a new construction of such
surfaces of degree $d=6$. Finally we provide some new lower bounds
for geometric genera of surfaces lying on general hypersurfaces of
degree $3d\ge 15$ in $\PP^4$.
\end{abstract}

\maketitle

\tableofcontents

\vfuzz=2pt
\thanks{}

\section*{Introduction}
What is the lowest geometric genus $\eta(n,d)$ of a reduced,
irreducible curve on a very general hypersurface of degree $d$ in
$\PP^n$? The case $n=2$ is trivial. For $n=3$ one has
\be\label{in0} \qquad \eta(3,d)=0\quad\mbox{if}\quad d\le
4\quad\mbox{while}\quad \eta(3,d)={d-1 \choose 2}-3
\quad\mbox{if}\quad d\ge 5 \,\ee and for any $d\ge 6$ this bound
is achieved by tritangent plane sections, and only by these
\cite{Xu1}. Similarly, $\eta(4,d)=0$ if $d\le 5$, while
$\eta(4,6)\ge 2\,$ \cite{ClR}. More generally, for $n \geq 4$ one
has
\[\,\quad\quad \eta(n,d)=0\quad \forall d\le
2n-3\qquad\mbox{and}\quad \eta(n,d)\ge 1\quad \forall d \geq
2n-2\,
\]
see \cite{Vo1} (in the case $d=2n-3$ see also \cite{Ha, Pac1,
Sh}).  Presumably, $\eta(n,d)\to\infty$ as $d\to\infty$, however,
the asymptotic of $\eta(n,d)$ is unknown. One is equally
interested in bounds for the geometric genus or other numerical
invariants of higher dimensional subvarieties in general
hypersurfaces, see e.g.\ \cite{CLR, Ei1, Ei2, Pac2, Vo2, Wa2,
Xu3}.

A projective variety $X$ is {\em algebraically hyperbolic} if it
does not admit a non--constant morphism from an abelian variety.
If there is an algebraically hyperbolic hypersurface of degree $d$
in $\PP^n$, then a very general hypersurface of degree $d$ is
algebraically hyperbolic as well. For instance, a very general
surface $X$ of degree $d\ge 5$ in $\PP^3$ is algebraically
hyperbolic. Indeed, $X$ does not contain rational or elliptic
curves since $\eta(3,d)\ge 3$ for $d\ge
5$ by \eqref{in0}. This also follows from Proposition \ref {201}
below if $d\ge 6$, while Corollary \ref {agthm} offers a short
proof of Xu's and Voisin's result about non-existence of rational
curves on a very general quintic in $\PP^ 3$. Since $X$ is of
general type it cannot be dominated by an abelian variety.

Similarly, a general sextic threefold $X$ in $\PP^4$ is
algebraically hyperbolic. Indeed, $X$ does not contain rational or
elliptic curves since $\eta(4,6)\ge 2$. By
\cite[Theorem 1]{Xu3} it also does not contain surfaces with
desingularization of geometric genus at most $2$. Therefore, every map
from an abelian variety to $X$ is constant.

A variety $X$  is  \emph{Kobayashi hyperbolic}, or simply
\emph{hyperbolic}, if it does not admit any non--constant entire
curve $\C\to X$. Hyperbolicity implies algebraic hyperbolicity,
and it is stable under small deformations.

Given one of the two above hyperbolicity notions,
one can ask what is the lowest degree $d=d(n)$ such that
a very general projective hypersurface in $\PP^n$ of degree $d$
possesses this property. For instance, the classical \emph{Kobayashi
problem} suggests that a very general hypersurface of degree $d\ge
2n-1$ in $\PP^n$ is hyperbolic.

It is known that, indeed, a very general surface of degree $d\ge
18$ in $\PP^3$ is hyperbolic \cite{Pau} (see also \cite{DEG, MQ}).
The existence of hyperbolic surfaces in $\PP^3$ of degree $d$ for
all $d\ge 8$ was established with a degeneration argument in
\cite{SZ2} (see the references in \cite{SZ2} for other
constructions), and for $d=6$ in \cite{Du}. In \S \ref
{ssec:d=67} below (see, in particular, Theorem \ref {mthm}) we
give an alternative proof for the case $d=6$, which works also in
the (so far unknown) case $d=7$. The case $d=5$ in the Kobayashi
problem for $\PP^3$ remains open.

Our method consist in degenerating a general hypersurface to a
certain special one, following the limits in the degeneration of
entire curves or of algebraic curves or surfaces, according to the
hyperbolicity notion we are dealing with. In this framework the
concept of Brody curves and their limits is very useful, cf.\ e.g.
\cite{SZ1, SZ2, Za1, Za2}. We recall a minimum of basics on this
subject in \S \ref {sec:brody}. Our preferable degenerations here
are to \emph{scrolls}, and we recall their main properties in \S
\ref {sec:scrolls}. In subsection \ref {ssec:ah} we give an easy
proof of non--existence of curves of low genera on very general
surfaces in $\PP^ 3$ of a given degree. In \S \ref {sec:higher} we
treat the higher dimensional case. In particular, in Theorem \ref
{thm:pg} we provide a lower bound for the geometric genus of
surfaces contained in very general hypersurfaces of degree $3d\ge
15$ in $\PP^ 4$.

 By a well-known theorem of Bogomolov  \cite{Bo}, on a smooth
surface $S$ of general type with $c^2_1 (S) > c_2(S)$, the curves
of a fixed geometric genus vary in a bounded family. This was
partially extended in \cite{LM} to any smooth surface $S$ of
general type by showing that there are only a finite number of
rational and elliptic curves on $S$ with a fixed number of nodes
and ordinary triple points and no other singularities.  In
subsection \ref {ssec:bound} we address the question whether
curves of a given geometric genus have bounded degree on a general
surface of degree $d\ge 5$ in $\PP^ 3$. We give an affirmative
answer for all genera $g\le d^ 2+O(d)$.

Finally in subsection \ref {ssec:d=67} we prove the aforementioned
Theorem \ref {mthm}.

\section{Scrolls}\label{sec:scrolls}

 \subsection{Generalities on scrolls}\label{ssec:general}
By a \emph{scroll} in $\PP^n$ we mean the image $\Sigma=\phi(S)$
of a smooth, proper $\PP^1$-bundle $\pi:S\to E$ under a birational
morphism $\phi:S\to \Sigma\hookrightarrow\PP^n$ which sends the
\emph{rulings} of $S$ (i.e. the fibres of $\pi$) to projective
lines, called \emph{rulings} of $\Sigma$. The variety $E$ is
called the \emph{base} of the scroll. We will denote by $H$ and
$F$  a hyperplane section and a ruling of $\Sigma$, respectively.
We may abuse notation denoting by $H$ and $F$ also their proper
transforms on $S$.

The induced morphism $\mu: E \to \Gr(1,n)$ to the Grassmanian of
lines in $\PP^n$ is birational onto its image. Any such morphism
$\mu$ appears in this way, where $\pi:S\to E$ is induced via $\mu$
by the tautological $\PP^1$--bundle over the Grassmanian.
Furthermore, $d=\deg(\Sigma)$ is equal to the degree of the
subvariety $\mu(E)$ under the Pl\"ucker embedding of the
Grassmanian \cite[12.4]{BCGMB}, \cite[11.4.1]{Do}.

We will suppose form now on that $\phi: S\to \Sigma$ coincides
with the normalization morphism. We denote by $\br(\Sigma)$ the
set of \emph{multibranch points} of $\Sigma$, i.e. the set of
points $x\in \Sigma$ such that $\phi^ {-1}(x)$ consists of more
than one point. If $x\not\in \br(\Sigma)$, e.g. $x$ is a smooth
point of $\Sigma$, then there is just one ruling passing through
$x$. Since $\phi$ is finite,  there is no point on $\Sigma$ which
belongs to infinitely many rulings.

We let $\Delta_\Sigma=\overline{\br
(\Sigma)}\subseteq\Sigma$ and
$\Delta_S=\phi^{-1}(\Delta_\Sigma)\subseteq S$.
We will assume that the following conditions hold:
\begin{itemize}
\item [(C1)] ~ $\dim(\Sigma)=n-1$;
\item [(C2)] ~ $\Delta_\Sigma$ coincides with $\Sing (\Sigma)$;
\item [(C3)] ~ $\Delta_\Sigma$  and $\Delta_S$
are both irreducible of dimension
$n-2$;
\item [(C4)] ~  a general point  $x\in \Delta_\Sigma$
is a normal crossing double point of $\Sigma$. In particular,
$\phi^ {-1}(x)$ has cardinality 2, and $x$ sits on two different
rulings;
\item [(C5)] ~ $\Delta_\Sigma$ contains no ruling, i.e.
$\mu: E \to \Gr(1,n)$
is injective.
\end{itemize}
In this situation $\Delta_\Sigma$ and $\Delta_S$ both have natural
scheme structures, and $\Delta_S$  is a reduced divisor on $S$.

Conditions (C1)--(C4) are verified if $S\subseteq\PP^{n+k}$ is a
smooth scroll of dimension $n-1$ and $\phi:S\to\Sigma$ is induced
by a general linear projection $\PP^{n+k}\dashrightarrow\PP^{n}$;
see \cite {Fr}. The last condition (C5) can be easily checked by
induction; we leave the details to the reader.

\blem\label{199} In the above setting, a general ruling of
$\Sigma$ meets the double locus $\Delta_\Sigma$ in $d-n+1$ points.
In particular, $\Sigma$ is swept out by an $(n-2)$--dimensional
family of $(d-n+1)$--secant lines of $\Delta_\Sigma$.\elem

\bproof  By the Ramification Formula \cite[9.3.7(b)]{Fu}
there is a linear equivalence relation on $S$
 \be\label{f201} \Delta_S \sim
(d-n-1)H-K_S\,.\ee Since $F\cdot H=1$ and $F\cdot K_S=-2$, we have
$F\cdot\Delta_S=d-n+1$. Since $\Delta_S$ is reduced, the general
ruling of $S$ meets $\Delta_S$ in $d-n+1$ distinct points.
The assertions follow because $\phi$ induces an isomorphism of
 each ruling of $S$ to its image.
\eproof

\subsection{Surface scrolls with ordinary singularities}
We restrict here to the case $n=3$. So $E$ is a smooth curve
of genus $g$ and $S\subseteq\PP^{3+k}$ and
$\Sigma\subseteq\PP^3$ are surfaces, called \emph{scrolls of
genus} $g$: here $g$ is  the \emph{sectional genus}
of the scroll.

\begin{rem}\label{rem:genus}
 For an irreducible curve $C$ on $S$ of genus $g'$ such that
$C\cdot F=\nu$, the Riemann--Hurwitz Formula implies the
inequalities $g'\ge \nu(g-1)+1\ge g$. In particular, for $g\ge 1$
the only irreducible curves on $S$ of geometric genus $g'<g$ are
the rulings, and for $g'=g\ge 2$ the curve $C$ is a
\emph{unisecant} i.e., the intersection number $\nu$ of $C$ with
rulings is 1. The same holds on $\Sigma$.
\end{rem}

 We say that $\Sigma$ has \emph{ordinary singularities} if, in
addition to conditions (C1)--(C5), the following hold:
\begin{itemize}
\item [(C6)]  the singularities of the \emph{double curve} $\Delta_\Sigma$
consist of finitely many triple points, which are  also ordinary
triple points of the surface $\Sigma$ (these are locally
analytically isomorphic to the surface singularity $xyz=0$ in
$\mathbb C^ 3$ at the origin);
\item [(C7)]  the non--normal crossings  singularities of $\Sigma$
are finitely many \emph{pinch points}. These are the points in
$\Delta_\Sigma\setminus\br(\Sigma)$, and there is just one ruling
through each of them. A pinch point has just one preimage on $S$,
which, abusing terminology, we will also call a pinch point;
\item [(C8)] the only
singularities of $\Delta_S$ are ordinary double points, three of
them over each triple point of $\Delta_\Sigma$. Furthermore, the
degree two map $\phi: \Delta_S\to \Delta_\Sigma$ is ramified
exactly over the pinch points of $\Sigma$.
\end{itemize}
These are the singularities of a general projection to $\PP^3$ of
a smooth surface in $\PP^ 4$, or even of a surface in $\PP^ 4$
with finitely many \emph{nodes}, i.e. double points with tangent
cone formed by two planes spanning $\PP^ 4$.
In this case the curves $\Delta_\Sigma$ and
$\Delta_S$ are irreducible, except for the projection in $\PP^ 3$
of the Veronese surface of degree $4$ in $\PP^ 5$ (cf.\ \cite
{Fr1, Fr2, MP, Mo}). Note that a
general projection to $\PP^4$ of any smooth surface in $\PP^{r}$ (with $r>4$)
has only nodes as singularities
and the Veronese surface of degree $4$ in $\PP^ 5$ is the only one
whose general projection to $\PP^ 4$ is smooth (see \cite {Sev, Zak1}).

The basic invariants of $S$ are
\[c_1^2=K_S^2=8(1-g),\quad c_2=e(S)=4(1-g),\quad\mbox{and}\quad
\chi(\cO_S)=\frac{c_1^2+c_2}{12}=1-g\] (see \cite{En}, \cite [Ch.
5, \S 2]{Ha}).  The following \emph{projective invariants} are
also important \be\label{401}
\begin{aligned}
\delta_\Sigma\quad=&\quad \deg(\Delta_\Sigma)\cr
\gamma_\Sigma\quad=&\quad \text{the geometric genus of}\quad
\Delta_\Sigma\cr t_\Sigma\quad=&\quad\text{the number  of triple
points of} \quad\Delta_\Sigma\cr p_\Sigma\quad=&\quad\text {the
number of pinch points of} \quad\Sigma\cr \tilde
\gamma_\Sigma\quad=&\quad \text{the geometric genus of}\quad
\Delta_S\cr
\end{aligned}
\ee (in the sequel we suppress the index $\Sigma$ when unnecessary).

For the proof of the following formulas see e.g.\ \cite{Bo},
\cite[\S 11.5]{Do}, \cite[p.\ 176]{En}, \cite{Pi},
\cite[(1)-(10)]{SZ0}, and references therein.

\bprop\label{300} Let $\Sigma$ stands as before for a scroll in
$\PP^3$ of degree $d$ and genus $g$ with ordinary singularities.
Then the projective invariants of $\Sigma$ are given by the {\rm
Bonnesen's formulas}

 \be\label{f100}
\delta={d-1\choose 2}-g\,,\ee

\be\label{f200}\gamma={d-3\choose 2}+(d-5)g\, ,
\ee

\be\label{f300} t={d-2\choose 3}-(d-4)g\,,\ee

\be\label{f400} p=2d+4(g-1)\,,\ee

\be\label{f500}
\tilde \gamma=2(\gamma+g)+d-3\,.\ee
\eprop

\brem\label{001} Due to (\ref{f300}),  for $d\ge 5$ the inequality
$t\ge 0$ reads $g\le\frac{1}{6}(d-2)(d-3)\,$. This implies
\be\label{f600} g\le d-4, \quad\text{ if}\quad d=5,6,7\,.\ee In
the sequel we also need the inequality \be\label{002} \gamma >
3(g-1)\qquad\text {for all}\quad g\ge 1\quad\text{and}\quad  d\ge
5\,.\ee This follows from (\ref{f200}) for $d\ge 8$ and from
(\ref{f200}) and (\ref{f600}) for $d=5,6,7$ (actually, $\gamma
>3g$ for all $g\ge 1$ and $d\ge 5$ except for $g=2$, $d=6$). \erem

\subsection{Surface scrolls with general moduli}
We recall a result from \cite{APS} (cf. also
\cite[Theorem 1.2]{CCFMLincei}).

\begin{thm}\label{thm:scrolls}
Let $g \geq 0$ be an integer and let $k = {\rm min}\{1, g-1\}$. If
$ d \geq 2 g + 3 + k$, then there exists a unique irreducible
component $\HH_{d,g}$ of the Hilbert scheme of scrolls of degree
$d$ and sectional genus $g$ in $\PP^{r}$, where
$r=d-2g+1$,
such that the general point $[S] \in
\HH_{d,g}$ represents a smooth scroll $S$ with $h^ 1(S, \mathcal
O_S(1))=0$, i.e. $S$ is \emph{non--special}. Furthermore
$\HH_{d,g}$ dominates the moduli space ${\mathcal M}_g$ of smooth
curves of genus $g$ via the map sending a scroll to its base.
\end{thm}

\begin{rems}\label{rem:3space}  (i) Assuming that $ d \geq 2 g + 3 + k$
(as in the above theorem), we have
$r\ge 3$ if $g=0$, $r\ge 4$ if $g=1$, and $r \ge 5$ if $g\ge 2$,
and we can project smooth scrolls $S$ with
$[S] \in \HH_{d,g}$ thus obtaining scrolls $\Sigma$  in $\PP^ 3$
with ordinary singularities and irreducible double curve.

(ii) The assumption of Theorem \ref{thm:scrolls} gives $d\ge 2g+4$
for $g\ge 2$ and $d\ge 2g+3=5$ for $g=1$. In fact, similar results
hold also for $g\ge 2$ and $d=2g+3$ or $d=2g+2$, while the
corresponding scrolls are no longer smooth.

More precisely, let $g\ge 2$ and $d=2g+3$ (i.e., $r=4$). Then
$\HH_{d,g}$ is a component of the Hilbert scheme, whose general
point $[S'] \in \HH_{d,g}$ represents a scroll $S'\subseteq \PP^
4$ with only nodes as singularities and
with a smooth normalization $S$ such that $h^ 0(S, \mathcal
O_S(1))=5$ and  $h^ 1(S, \mathcal O_S(1))=0$
(this can be shown with
the same analysis as in \cite{CCFMLincei}). Once again,
$\HH_{d,g}$ dominates the moduli space ${\mathcal M}_g$.

If  $g\ge 2$ and $d=2g+2$ (i.e., $r=3$), a similar assertion
holds. However, now $\HH_{d,g}$ is no longer a component of the
Hilbert scheme, but a locally closed subset of the projective
space $\mathcal L_d=\vert \mathcal O_{\PP^ 3}(d)\vert$ of all
surfaces of degree $d$ in $\PP^3$. It is reasonable to expect that
a general point $[\Sigma] \in \HH_{d,g}$ represents a scroll
$\Sigma\subseteq \PP^ 3$ with ordinary singularities. This would
follow by going deeper into the analysis performed in
\cite{CCFMLincei}, but we do not use this here in the full
generality. We investigate below in more detail various examples
(see especially Example \ref{sextic}).
\end{rems}

\bexa\label{quartic} {\em Elliptic quartic scrolls.} Let $E$ be
a smooth curve of type $(a,b)$ on $\PP^1\times\PP^1$,  identified
with a smooth quadric in $\PP^ 3$. The genus of $E$ is
$g=ab-a-b+1$. Consider a
pair of skew lines $R_1, R_2$ in $\PP^ 3$. Identifying these lines
with the factors of $\PP^1\times\PP^1$, we can interpret the
canonical projections of $E$ to the factors as maps $\phi_i: E\to
R_i$, $i=1,2$, of degree $a$ and $b$, respectively. For each $x\in
E$ we consider the line $L_x$ joining  the points
$\phi_i(x)$, $i=1,2$. This yields the map $\mu: x\in E\to L_x\in \Gr(1,3)$.
Its image is a smooth curve on $\Gr(1,3)$
under the Pl\"ucker embedding of the Grassmanian
 $\Gr(1,3)$ as a quadric in $\PP^5$. The associated scroll
\[\Sigma=\Sigma_{a,b}=\bigcup_{x\in E} L_x\]
in $\PP^3$ with base $E$ has degree $a+b$. Indeed, it has
singularities of multiplicities $a$ along $R_1$ and $b$ along
$R_2$. So a line $\langle A,B\rangle$, where $A\in R_1$ and $B\in
R_2$, meets $\Sigma$ only in $A$ and $B$.

In particular, for $a=b=2$ we obtain a quartic scroll in $\PP^3$
of genus $1$ with two skew double lines, and for $a=3,\, b=2$ a
quintic scroll of genus $2$ with a double line and a triple line.

From now on, we concentrate on an elliptic quartic scroll
$\Sigma=\Sigma_{2,2}$. The preimage $\Delta_S$ of $\Delta_\Sigma$
on $S$ consists of two disjoint copies $E_1,E_2$ of $E$ with
$\phi_i: E_i\to R_i$, $i=1,2$, corresponding to two distinct
$g_2^1$'s on $E$. There are in total 8 pinch points of
$\Sigma$, 4 on each of the lines $R_1, R_2$. These are the branch
points of the maps $\phi_i$, $i=1,2$. If these maps are
sufficiently general, also the pinch points are generically
located along $R_1, R_2$ and the ruling passing through a pinch
point does not contain any other pinch point.

Let us illustrate on this example our degeneration method. Any
smooth elliptic quartic curve is a complete intersection of two
quadrics in $\PP^ 3$. Hence it embeds as well to the Grassmanian
$\Gr(1,3)$. By virtue of Remark \ref{rem:3space} to Theorem \ref
{thm:scrolls} (the case $r=3$) these curves fill in a unique
irreducible component $\mathcal H_{4,1}$ of the Hilbert scheme of
curves of degree $4$ in $\Gr(1,3)$, which dominates the moduli
space $\mathcal M_1$. The component ${\mathcal H_{4,1}}$ contains
all \emph{limit curves}, e.g. all reduced, nodal curves of degree
4 and arithmetic genus 1 spanning a $\PP^ 3$. For instance, the
union $E_0$ of two conics $\Gamma_1,\Gamma_2$ meeting
transversally at two distinct points $f_1,f_2$ is such a limit
curve. The curve $E_0$ corresponds to the union $\Sigma_0$ of two
quadrics surfaces $Q_1,Q_2$ in $\PP^3$ associated to the conics
$\Gamma_1, \Gamma_2$ on the Grassmanian $\Gr(1,3)$. We may assume
these quadrics to be smooth. They intersect along the
quadrilateral $F_1\cup F_2\cup G_1\cup G_2$, where the lines $F_1,
F_2$ correspond to $f_1, f_2$ and belong to the same ruling on
each quadric, and $G_1, G_2$ are distinct lines belonging to the
other ruling. We let $p_{ij}=F_i\cap G_j$, $i,j=1,2$.

The surface  $\Sigma_0$ can be seen as a flat limit of surfaces of
type $\Sigma$, since it corresponds to a point in
${\mathcal H_{4,1}}$. The limit of the ruling of
$\Sigma$ is the union of the two rulings of $Q_1$ and $Q_2$
containing $F_1,F_2$. The limits of the double lines $R_1, R_2$
are the lines $G_1,G_2$. The limit of each of the components $E_i$
of the curve $\Delta_S$ on $S$ consists of two copies of $G_i$
glued at $p_{1i}, p_{2,i}$. Each of these points is the limit of
two pinch points of $\Sigma$.

Conversely, when we deform $\Sigma_0$ to $\Sigma$, the two double
lines $F_1$ and $F_2$ of $\Sigma_0$ disappear, because we are
smoothing the two nodes of $E_0$. Each of the points $p_{ij}$
($i,j=1,2$) gives rise to two pinch points generically located
along the double line of $\Sigma$, which deforms $G_j$. \eexa

\bexa\label{quintic} {\em Elliptic quintic scrolls.} Consider now
the case where $d=5$ and $ g=1$. By Theorem \ref{thm:scrolls}, a
general point $[S] \in \HH_{5,1}$ represents a smooth scroll $S$
in $\PP^ 4$, whose general projection $\Sigma$ to $\PP^ 3$ has
ordinary singularities. According to Bonnesen's formulas
(\ref{f100})-(\ref{f500}), the double curve $C=\Delta_\Sigma$ is
an irreducible, smooth, elliptic quintic curve, which contains the
$10$ pinch points of $S$. Its preimage $\tilde C=\Delta_S$ is a
smooth, irreducible curve on $S$ of genus $6$. By Lemma \ref
{199}, the rulings of $\Sigma$ are trisecant lines to $C$.

Conversely, for any smooth elliptic quintic curve $C$ in $\PP^3$,
the trisecant lines to $C$ sweep out a quintic scroll $\Sigma$,
which is singular exactly along $C$ (cf.\  Berzolari's Formula,
Proposition 1 and Corollary 2 in \cite{MBer}). Such a surface
$\Sigma$ is an elliptic scroll, and by the Riemann--Roch Theorem
it comes as a projection of a surface represented by a point in
$\HH_{5,1}$ as above.

Any such scroll $\Sigma$ corresponds to an embedding of an
elliptic quintic  curve $E$ in $\Gr(1,3)$ via the map $\mu$ as in
\S \ref {ssec:general}. The image of $E$ is a quintic  elliptic
normal curve, contained in a hyperplane section of $\Gr(1,3)$.
Indeed, any normal, elliptic quintic  curve lies on some smooth
quadric in $\PP^ 4$, hence on a hyperplane section of
$\Gr(1,3)$.

There is another interpretation of these elliptic quintic scrolls.
Let $E$ be an elliptic curve. Consider its symmetric product
$E(2)$, formed by all degree 2 effective divisors on $E$. The
class of the \emph{diagonal} $D=\{2p, p\in E\}$ is divisible by
$2$ in ${\rm Pic}(E(2))$; we denote  by $\vartheta$ the class of
its half. One has $K_{E(2)}\sim -\vartheta$.

The Abel--Jacobi map $\alpha: E(2)\to {\rm Pic}^ {(2)}(E)\cong E$
makes $E(2)$ a $\PP^ 1$--bundle with base $E$. The rulings are the
$g_2^ 1$'s on $E$. The \emph{coordinate curves} $E_p=\{x+p, x\in
E\}\cong E$ are unisecant curves of the rulings and form a
one--dimensional family parametrized by the point  $p$ varying on
$E$. We have $E_p^ 2=1$. If $F_1, F_2$ are rulings, then the
divisor class of the curve $E_p+F_1+F_2$ is very ample on $E(2)$
and maps isomorphically the surface $E(2)$ onto a quintic scroll
$S$ in $\PP^ 4$.  Each coordinate curve $E_p$ is  mapped to a
smooth plane cubic on $S$ which is the residual intersection of
$S$ with a hyperplane containing two rulings. Conversely any
smooth plane cubic on $S$ is a coordinate curve: indeed, it sits
on a 1--dimensional family of hyperplane sections of $S$ and their
residual intersections with $S$  is a pair of lines.

Let as before  $\Sigma$ denote the image of $S$ under a general
projection $\PP^4\dashrightarrow\PP^ 3$. Any coordinate curve on $S$
is isomorphically mapped to a smooth plane cubic and
the images on $\Sigma$ of two distinct
coordinate cubics on $S$ are distinct. This provides a complete,
one--parameter family of smooth plane cubic curves on $\Sigma$
which are the only plane cubics on $\Sigma$.
Let $L$ be the plane containing one of them $\bar E$.
The residual intersection  on
$L\cap\Sigma$ must be a union of two rulings, which meet on $C$.
The corresponding rulings on $S$ span a hyperplane
which cuts out on $S$ a coordinate cubic $\tilde E$
plus the two rulings. Hence $\bar E$ is the image of $\tilde
E$ on $\Sigma$.

Let $x\in C$ be a general point and $F_1,F_2$ the two
rulings through $x$. The plane $\pi$ spanned by them cuts
 $\Sigma$ in the union of $F_1,F_2$ and a smooth cubic $\bar E$,
 which is the projection of a unique coordinate curve. When $x$ varies,
 we obtain in this way all projections of coordinate curves.
 This shows that $C$ is isomorphic to $E$, since it parametrizes
 the family of coordinate curves.

 When the center of projection $\PP^4\dashrightarrow\PP^ 3$
varies we obtain a monodromy action. The following argument shows
that this monodromy is irreducible on appropriately chosen
objects.

 The cubic curve $\bar E$ as above does not pass through
$x$, and cuts the ruling $F_i$ in three (generically distinct)
points $p_i, q_{i1},q_{i2}$, $i=1,2$, such that $q_{i1},q_{i2}\in
C$. Indeed, $p, q_{i1},q_{i2}$, $i=1,2$, are the five intersection
points of $\pi$ with $C$. By moving the centre of projection, we
may assume that the pair of rulings $(F_1,F_2)$ corresponds to a
general divisor of a given $g^ 1_2$ on $E$, and that
$q_{11}+q_{12}$ ($q_{21}+q_{22}$, respectively) is a general
divisor in the $g^ 1_2$ cut out on $\bar E$ by the lines through
$p_1$ (through $p_2$, respectively). In conclusion, by moving the
centre of projection the monodromy interchanges the pairs
$q_{11}+q_{12}$ and $q_{21}+q_{22}$ and also interchanges the
points in each pair separately.
  \eexa

 \bexa\label{sextic} {\em Sextic scroll of genus two.}
 By the case $g\ge 2,\,r=3$ of Remark \ref {rem:3space}.2,
 there exist sextic scrolls $\Sigma$
 of genus two in  $\PP^ 3$.
 They correspond to genus 2 curves of degree 6
 on the Grassmanian $\Gr(1,3)$. In fact,
 by the Riemann--Roch Theorem,  any smooth curve of genus 2
 embeds in  $\PP^ 4$ as a sextic. This sextic spans $\PP^ 4$
 and lies on
 a smooth quadric in $\PP^ 4$, hence on a hyperplane section of
 the Grassmanian $\Gr(1,3)$.
 These curves fill in a unique component $\mathcal H_{6,2}$
 of the Hilbert
 scheme of curves of degree $6$ and genus 2 in $\Gr(1,3)$,
 which dominates
 $\mathcal M_2$ via the natural map.
 As in Example \ref {quartic}.2,
 ${\mathcal H_{6,2}}$ contains
{limit curves}, and in particular all reduced, nodal curves of
degree 6 and arithmetic genus 2 spanning a $\PP^ 4$.

Assuming that a general such scroll $\Sigma$ has ordinary
singularities, Lemma \ref {199} and Proposition \ref {300} say
that the rulings of $\Sigma$ are four--secant lines to the double
curve $C=\Delta_\Sigma$, which is a smooth, irreducible curve  in
$\PP^3$ of degree $8$ and genus $5$, passing through all $16$
pinch points of $\Sigma$. The preimage $\tilde C=\Delta_S$ of $C$
on $S$ is a smooth curve of degree $16$ and of genus $17$.

Let us show that a general sextic scroll $\Sigma$ in $\PP^3$ of
genus 2 has ordinary singularities and an irreducible double curve
$C$. Consider a reducible sextic curve $E_0\subseteq \PP^ 4$ of
arithmetic genus $2$, which consists of a general smooth elliptic
normal quintic curve $E'$ and a line $D$ meeting $E'$
transversally in two distinct points. Such a curve $E_0$
corresponds to a point in ${\mathcal H_{6,2}}$, hence to a
reducible surface $\Sigma_0$, which is a limit of genus 2 sextic
scrolls $\Sigma$. On the other hand, $\Sigma_0$ is the union of a
general quintic elliptic scroll $\Sigma'$ in $\PP^3$ arising from
$E'$ as in Example \ref {quintic} plus a plane $\pi$ through the
two rulings $F_1,F_2$ of $\Sigma_0$, which correspond to the
intersection points of $E'$ and $D$. These rulings meet at a point
$p$ of the double curve $C'$ of $\Sigma'$.  The ruling on $\pi$ is
given by the pencil of lines passing through $p$, which
corresponds to the line $D$. The plane $\pi$ cuts out on $\Sigma'$
the union of the rulings $F_1$, $F_2$ and a smooth plane cubic
$\bar E$, as described in Example \ref {quintic}. The
singularities of $\Sigma'$ consist of $C'$, $F_1$, $F_2$, and
$\bar E$.

When we deform $E_0$ to a general smooth sextic $E$ on $\Gr(1,3)$,
the scroll $\Sigma_0$ is deformed to an irreducible sextic  scroll
$\Sigma$. The double lines $F_1$ and $F_2$ of $\Sigma_0$
disappear, because we are smoothing the two nodes of $E_0$. This
means that the flat limit on $\Sigma_0$ of the singular locus of
$\Sigma$ is the nodal curve $C_0=C'\cup\bar E$ of arithmetic
genus $5$. Hence $\Sigma$ is singular only along a double curve
$C$, which has arithmetic genus $5$. The latter curve is
irreducible. Indeed, otherwise this would be still a union of the
form $C'\cup\bar E$, and so the four-secant lines to $C$ would
sweep out a union of an elliptic scroll and a plane. However, this
is impossible since $\Sigma$ is irreducible and swept out by the
four-secants of the double curve $C=\Delta_\Sigma$.

 Since $C_0$ is nodal so is
 $C$. We claim that $C$ is actually smooth. Indeed, we may
 restrict our family to a general irreducible curve germ in
 ${\cH_{2,6}}$
 through
$\Sigma_0$, and then normalize this germ. In this way we obtain a
family of sextic scrolls over the disc $\D$ with a family
$\cC\to\D$ of double curves. The central fibre of $\cC$ is a
reducible nodal curve $C_0=C'\cup\bar E$ with $4$ nodes. Assuming
that no one of these nodes is smoothed on a general fibre $C$ of
the family, $C$ should also have $4$ nodes. These nodes represent
an \'etale four-sheeted cover over the disc. Now we can normalize
the fibres of the family $\cC$ simultaneously (see e.g.,
\cite{Se}), thus obtaining a smooth family with an irreducible
general fibre and a disconnected cental fibre. The latter
contradicts the \emph{Connectedness Principle} (see \cite [Ch.
III, Ex.\ 11.4, p.\ 281] {har}).

Consequently,
 at least one of the four nodes of $C_0$ has to be smoothed
 in the deformation to $C$.
 But then by the irreducibility of the monodromy
 (see the final part of Example \ref {quintic}) all nodes of
 $C_0$ have to be smoothed. \eexa

\section{Bounding degrees of low genera curves on surfaces}
\subsection{Algebraic hyperbolicity} \label{ssec:ah} Scrolls can be
used to establish algebraic hyperbolicity of very general surfaces
of a given  degree $d$ in $\PP^3$. For $d\ge 6$ this is done in
Proposition \ref{201} below. In the proof we use the Albanese
inequality (see \cite {Alb, No} (see also \cite[\S
4(b)]{Morr}), which says the following: if a reduced projective
curve $C$ of geometric genus $g$ degenerates into an effective
cycle $C_0=\sum_i m_iC_i$, where $C_i$ is a reduced projective
curve of geometric genus $g_i$, then \be\label{alb} g\ge
\sum_{g_i\ge 1} (m_i(g_i-1)+1)\,.\ee In particular, $m_i(g_i-1)\le
g-1$ if $g_i\ge 1$. So $g_i\le g$  for all $i$.

\bprop\label{201} Assume that there exists a scroll $\Sigma$ of
degree $d\ge 5$ and genus $g\ge 1$ in $\PP^3$ with ordinary
singularities. Then a very general surface $X$ in $\PP^3$ of
degree $d$ does not contain curves of geometric genus $g'<g$.
\eprop

\bproof Let $X$ be a very general surface  in $\PP^3$ of degree
$d$. By the Noether-Lefschetz Theorem, the Picard group of $X$ is
generated by $\mathcal O_X(1)$. Consider the pencil $\{X_t\}_{t\in
\PP^ 1}$ generated by $X_0=\Sigma$ and $X_\infty=X$. This gives
rise to a flat family of surfaces $f: \mathcal X\to \mathbb D$
over  a disc $\mathbb D$, where the {central fibre} over $0$ is
$X_0$,  all fibres $X_t$ with $t\in\D\setminus\{0\}$ are smooth
and ${\rm Pic}(X_t)$ is generated by $\mathcal O_{X_t}(1)$ for a
very general such fibre.  We claim that a very general surface of
this family does not contain any curve of geometric genus $g'<g$.
We argue by contradiction and assume that this is not the case for
some $g'< g$.

For each positive integer $n$ we may consider the locally closed
subset $\mathcal H_{n,g'}$ of the relative Hilbert scheme of $f:
\mathcal X\setminus X_0\to \mathbb D\setminus \{0\}$, whose points
correspond, for each $t\neq 0$, to the irreducible curves of
geometric genus $g'$ in $\vert \mathcal O_{X_t}(n)\vert$.  By our
assumption, there is a component of $\mathcal H_{n,g'}$ which
dominates $\D\setminus \{0\}$. Let $\mathcal H$ be the closure of
this component in the relative Hilbert scheme of $f: \mathcal X\to
\mathbb D$. By the properness of the relative Hilbert scheme,
$\mathcal H$  surjects onto $\mathbb D$. Hence there is a curve
$C_0\in  \mathcal O_{X_0}(n)$ on $X_0$, which corresponds to a
point in $\mathcal H$. By Albanese's inequality (\ref{alb}), every
component of $C_0$ has geometric genus $g''\le g'<g$. By \eqref
{f200} (for $g\ge 2$) and Example \ref{quintic} (for $g=1$) we
have $\gamma\ge g>g''$, where $\gamma$ stands as before for the
geometric genus of the double curve $\Delta_\Sigma$ of
$X_0=\Sigma$. Hence no component of $C_0$ coincides with
$\Delta_\Sigma$. Now the pull--back $\Gamma$ of $C_0$ on the
normalization $\phi: S\to \Sigma$ belongs to the linear system
$\vert \phi^ *(\mathcal O_\Sigma(n))\vert$ and maps birationally
to $C_0$ by the finite map $\phi$. Since the only curves of genus
smaller than $g$ on $S$ are rulings, $\Gamma$ consists of rulings.
In particular, $\Gamma^ 2=0$. On the other hand, since $\Gamma \in
\vert \phi^ *(\mathcal O_\Sigma(n))\vert=\vert \mathcal
O_S(n)\vert$ we have $\Gamma^2=n^2d>0$, a contradiction.
\eproof

In Proposition \ref{201.1} below we slightly strengthen Proposition \ref
{201}, using Proposition \ref{acz} and Corollary \ref{acz2}.

Keeping in mind Example \ref {quintic}, Proposition \ref{201}
provides an alternative quick proof of the following result
originally established by Xu \cite{Xu1} and Voisin \cite{Vo1,
Vo2}.

\bcor\label{agthm} On a very general surface of degree $d\ge 5$ in
$\PP^3$ there is no rational curve. \ecor

\emph{Very general} in Corollary \ref{agthm}
 can be replaced by \emph{general}
provided the following question is answered in negative.

\begin{quest}\label{212} {\em Does there exist a sequence of
smooth quintic surfaces $X_n$ in $\PP^3$ such that $X_n$ contains
a rational curve of degree $d_n$ and not smaller, with
$d_n\to\infty$?}
\end{quest}

\brem\label{213} Notice that for any integers $n \ge 3$,  $d > 0$
and $0\le \delta \le d^2(n-1)+1$, the linear system $\vert
\mathcal O_S(d)\vert$ on a general K3 surface $S$ of degree $2n-2$
in $\PP^ n$ with Picard group generated by $\mathcal  O_S(1)$,
contains a $(d^2(n -1)-\delta + 1)$--dimensional family of
irreducible $\delta$--nodal curves, whose geometric genus equals
$d^2(n -1)-\delta + 1$ (see \cite {chen}). So $S$ contains nodal
curves of every geometric genus $g\ge 0$. This applies in
particular to general quartic surfaces in $\PP^ 3$. \erem

\subsection{Bounding degrees of curves of low genera on general surfaces in
$\PP^3$}\label{ssec:bound} In this section we address the
following \emph{boundedness question} (cf.\ \cite{Bo, LM}
and the related discussion in the Introduction):

\begin{quest} \label{quest:bound} {\em Given integers $d\ge 5$ and $g\ge 0$,
does there exist a bound $n_{d,g}$ such that every irreducible
curve of geometric genus $g$ on a very general surface of degree
$d$ in $\PP^ 3$ has degree $n\le n_{d,g}$?}
\end{quest}

If $d=4$ the answer is negative (see \cite {chen, Hal} and Remark
\ref{213}). The argument in the proof of Propositions \ref {201}
and \ref {201.1} can be used to give an affirmative answer for
$d\ge 6$ and small enough $g$.

\bprop\label{prop:bound} Suppose there exists a scroll $\Sigma$ of
degree $d\ge 6$ and genus $g\ge 2$ with ordinary singularities.
Then the answer to Question \ref {quest:bound} is affirmative for
all genera $g'<\gamma$, where $\gamma$ is defined in (\ref{401}).
\eprop

\bproof We apply the same argument as in the proof of Proposition
\ref {201}. Keeping the notation of this proposition, we let again
$C_0\in \vert \mathcal O_\Sigma(n)\vert$ denote a curve which is a
limit of a flat family of irreducible curves $\{C_t\}_{t\in
\mathbb D-\{0\}}$, $C_t\in \vert \mathcal O_{X_t}(n)\vert$, of
genus $g'$, where $g'\ge g\ge 2$ by Proposition \ref{201}.
Write $C_0=m_1C_1+\ldots+m_hC_h+C'$ as a cycle, where for every
$i=1,\ldots, h$ the curve $C_i$ is irreducible of geometric
genus $g_i\ge 1$ and its transform on $S$ has positive
intersections $n_i$ with the rulings, whereas $C'$ consists of
rulings. Note that $n=\sum_{i=1}^ h m_in_i$. By Albanese's
inequality (\ref{alb}) and our hypothesis $g'<\gamma$, none of the
components of $C_0$ coincides with $\Delta_\Sigma$, and
\[
g'\ge h+  \sum_{i=1}^ h m_i(g_i-1)\,.
\]
The Riemann--Hurwitz formula yields: $g_i-1\ge n_i(g-1)$ for all
$i=1,\ldots, h$, so that  $\gamma> g'\ge h + n(g-1)$. This
provides a bound $n<(\gamma-1)/(g-1)$ (we remind that $g\ge
2$).
\eproof

\bcor\label{cor:bound} Question \ref {quest:bound} has an
affirmative answer for
\[
\begin{aligned}
&d=6, \quad g\le 5\,,\cr &d\ge 7 \quad {\text even}, \quad
g<(d-4)^2 \,,\cr &d\ge 7 \quad {\text odd}, \quad g<\frac
{(d-3)(2d-9)}2\,\,.\cr
\end{aligned}
\]
\ecor

\bproof For $d=6$ we use the sextic scroll of genus 2 as in
Example \ref {sextic}. For $d\ge 7$ even we write $d=2m+4$ and we
consider in $\PP^3$ general projections of smooth scrolls of genus
$m$ and degree $d$ in $\PP^ 5$ as in Theorem \ref {thm:scrolls}.
For $d\ge 7$ odd we write $d=2m+3$ and we consider general
projections of scrolls of genus $m$ and degree $d$ in $\PP^ 4$ as
in Remark \ref {rem:3space}.2. Applying Proposition \ref
{prop:bound} and taking into account \eqref {f200}, the assertion
follows. \eproof

\subsection{Families of  low degree curves of a given genus
on general surfaces in $\PP^3$}
 Proposition \ref{acz} below
extends a similar result by Arbarello--Cornalba \cite[Theorem
3.1]{ac2}, \cite {ac1} and Zariski \cite{za}; cf. also Knutsen
\cite[Lemma 4.4]{Kn}.

Let $S$ be a smooth projective surface,  Hilb$_1(S)$ the Hilbert
scheme of curves on $S$, and $\mathcal V_g(S)$ the locally closed
subset of Hilb$_1(S)$ formed by irreducible curves of geometric
genus $g$.

\begin {prop} \label{acz} In the setting as before,
for an irreducible component $\mathcal V$ of $\mathcal V_g(S)$ we
let $v=\dim(\mathcal V)$ and $\kappa=K_S\cdot \Gamma$, where a
curve $\Gamma$ on $S$ corresponds  to a general point in $\mathcal
V$. Then $v\le \max\{g,g-1-\kappa\}$. Furthermore, if $v>g$ then
$v=g-1-\kappa$, and the general curve $\Gamma$ of $\mathcal V$ has
only nodes as singularities.
\end{prop}

\begin{proof} Let $f: C\to \Gamma$ be the normalization.
The exact sequence
\[
0\to T_C\to f^ *(T_S)\to N_f\to 0
\]
defines the \emph{normal sheaf} $N_f$ to the map $f: C\to S$. It
can be included into an exact sequence
\[
0\to \tau \to N_f \to N'\to 0 \,,
\]
where $\tau$ is the torsion subsheaf of $N_f$ supported at the
points, where the  rank of the differential of $f$ drops, and $N'$
is an invertible sheaf.  Due to the {\em Horikawa inclusion}
$T_{[\Gamma]}(\mathcal V)\subseteq H^ 0(C,N')$ (see
\cite[(1.3)]{ac2} or \cite[Lemma 1.4]{ac1}) we have $v\le
h^0(C,N')$. By Riemann-Roch, $$h^0(C,N')=\deg (N')-g+1+h^1(C,N'),
\quad\text{where}\quad \deg(N')\le\deg(N_f)=2g-2-\kappa\,.$$ If
$h^1(C,N')=0$ this gives $v\le g-1-\kappa$. Otherwise $N'$ is
special, so $h^0(C,N')\le g$.
In any case, $v\le \max\{g,g-1-\kappa\}$, as stated.

If $v>g$ then  $h^ 1(C, N')=0$.
Since $H^1(C,\tau)=0$
this yields $H^ 1(C, N_f)=0$. As in \cite[proof of (1.5) and p.
96]{ac2} this implies $\tau=0$, hence $\Gamma$ is immersed
(i.e., has no cuspidal
singularities). One ends the proof as in \cite[pp. 96--98]{ac2}.
\end{proof}

For $\mathcal L_d=\vert \mathcal O_{\PP^ 3}(d)\vert$ we let
\be\label{400} N_d=\dim\,(\mathcal L_d)={{d+3}\choose 3}-1\,. \ee
Given a smooth surface  $X$ of degree $d$ in $\PP^ 3$ and
non--negative integers $n,g$, we let $\mathcal V_{n,g}=\mathcal
V_{n,g}(X)$ denote the locally closed subset of  $\mathcal
L_{X,n}=\vert \mathcal O_X(n)\vert$  formed by irreducible curves
on $X$ of geometric genus $g$. We also let
\[
g_{d,n}=\frac {dn(d+n-4)}2+1\,
\]
denote the arithmetic genus of the curves in $\mathcal L_{X,n}$.
Notice that $g_{d,n}=g+\nu$ if a general member of $\mathcal
V_{n,g}$ is nodal with $\nu$ nodes.

\bcor\label{acz2} Let $X$ be a general surface of degree $d\ge 3$
in $\PP^ 3$. If $g\ge 0$ and $n\in\{1,2\}$ are such that
$\mathcal V_{n,g}$ is nonempty, then
\[
 g_{d,1}-3 \le g\le g_{d,1}
\quad \text {if} \quad n=1 \quad \text {and} \quad g_{d,2}-9\le
g\le g_{d,2} \quad \text {if} \quad n=2\,.
\]
Furthermore, for every irreducible component $\mathcal V$ of $\mathcal
V_{n,g}$,  its general curve has exactly $\nu$ nodes as singularities and
its dimension is
\be\label{trahtibidoh} 3-\nu=
g-g_{d,1}+3 \quad \text {if} \quad n=1\quad \text {and}\quad
 9-\nu=  g- g_{d,2}+9
 \quad \text {if} \quad n=2\,.
 \ee
\ecor

\begin{proof} Let us show the assertion in the case $n=2$,
the case $n=1$ being similar. Consider the incidence relation
$I\subseteq\mathcal L_d \times \mathcal L_2$ consisting of all
pairs $(X,Q)$ such that $X$ is smooth and $Q$ and  $X$ intersect in an
irreducible curve $C$ of geometric genus $g$. Then $I$ is locally
closed and comes equipped with  the natural  projections $p: I\to
\mathcal L_d$ and $q: I\to \mathcal L_2$.

Note that if $(X,Q)\in I$ and $C$ is the intersection of $X$ and
$Q$, then we have a family of dimension $\dim(\mathcal L_{d-2})+1$
of pairs $(X',Q)\in I$ such that  intersection of $X'$ and $Q$ is
$C$: indeed we can take $X'$ general in the span of $X$ and of all
surfaces of degree $d$ containing $Q$.

By our assumption $p$ is dominant. Let $I'$ be an irreducible
component of $I$ which dominates $\mathcal L_d$ via $p$, so that
$\dim(I')\ge N_d$. We assume that $q(I')$ contains a smooth
quadric  $Q$  (the argument is similar otherwise, the details are
left to the reader). Then $I'$ dominates $\mathcal L_2$ via $q$
and we may assume $Q$ to be a general quadric. All components of
$q^ {-1}(Q)$ have dimension $\dim(I')-\dim(\mathcal L_2)$. Any
such component can be identified with a family of surfaces of
degree $d$. By the above discussion, the family of curves
$\mathcal V$ they cut out on $Q$ has dimension
\[v=\dim(I')-\dim(\mathcal L_{d-2})-\dim(\mathcal L_2)-1\,.\]
Moreover, $\mathcal V$ is an irreducible component of $\mathcal
V_{d,g}(Q)$. We have \[v\ge
N_d-N_{d-2}-N_2-1=g_{d,2}+4d-10>g_{d,2}\ge g.\] By Proposition
\ref {acz}, one has $v=g-1+4d$, which yields $g_{d,2}-g\le 9$.
Furthermore, by Proposition \ref{acz} the general curve in
$\mathcal V$ has at most nodes as singularities, which implies
(\ref{trahtibidoh}).
\end{proof}

 Corollary \ref {acz2}  could be extended to handle  also the case $n=3$.
 This requires however to analyze a number of cases,
 which we avoid here.

 Now we can strengthen Proposition \ref{201} as follows.

\bprop\label{201.1} Assume that there exists a scroll
$\Sigma$ of degree $d\ge 5$ and genus $g\ge 1$ in $\PP^3$ with ordinary
singularities.  Then a
very general surface $X$ of degree $d$ in $\PP^3$ does not contain
curves of geometric genus $g'\le 3(g-1)$.  \eprop

\begin{proof}
By  Proposition \ref{201}  we may suppose that $d\ge 6$ and $g'\ge
g\ge 2$. We proceed as in the proof of this proposition, using the
same notation. We argue by contradiction and assume that there is
a positive $g'\le 3(g-1)$, a positive integer $n$ and a component
of $\mathcal H_{n,g'}$ which dominates $\mathbb D\setminus \{0\}$.
Consider a curve $C_0\in\cO_{\Sigma}(n)$ as in the proof of
Proposition \ref {201}. As shown in this proof, $C_0$ cannot be
composed of rulings. Hence it contains a component $C_i$ of
geometric genus $g_i>0$, appearing in $C_0$ with multiplicity
$m_i$. By Albanese's inequality \eqref {alb}  one has $g'-1\ge
m_i(g_i-1)$. By (\ref{002}) and our assumption $g'\le
3(g-1)<\gamma$, hence $C_i\neq \Delta_\Sigma$. Therefore $C_i$
lifts birationally to the normalization $S$ of $\Sigma$ yielding a
$\nu_i$--secant of the ruling on $S$. Combining the inequalities
above, by Hurwitz Formula (see Remark \ref {rem:genus}) we obtain
$$3(g-1)-1\ge g'-1\ge m_i(g_i-1)\ge \nu_i m_i(g-1)\,.$$
Hence $\nu_i m_i\le 2$ and so the only
possibilities are
\[ \nu_i=m_i=1, \quad   \nu_i=1,\, m_i=2,\quad {\rm  and}\quad
\nu_i=2,\, m_i=1.\]
In the former case by \eqref {alb} there can be at most two such
components, while in the latter two cases at most one.
We have $n=\sum_i\nu_i m_i$, the sum over all components
$C_i$ of $C_0$ of positive genus.
It follows that $1\le n\le 2$. Then Corollary \ref
{acz2} yields $g'\ge g_{d,1}-3$, since $g_{d,1}-3<g_{d,2}-9$ for
$d\ge 6$. Thus we must have
\begin{equation}\label{eq:ineq}
\frac {(d-1)(d-2)}2-3=g_{d,1}-3\le g'\le 3(g-1)\le \frac
{d(d-5)}2\,,
\end{equation}
the last inequality coming from \eqref {f300}
for $d\ge 6$. But \eqref {eq:ineq}
gives a contradiction.\end{proof}

\section{Bounding geometric genera of divisors on general
$3$-folds in $\PP^4$}
\label{sec:higher} A simple way of constructing higher dimensional
scrolls consists in starting with the trivial $\PP^1$--bundle
$\pi:S=E\times\PP^1\to E$ over a smooth projective variety
$E\subseteq\PP^m$ of degree $d$ and dimension $n$. Let
$\Seg_{a,b}$ denote the image of $\PP^a\times\PP^b$ via the
\emph{Segre embedding}. Then
$$S\hookrightarrow\PP^m\times\PP^1\stackrel{\simeq}
{\longrightarrow}\Seg_{m,1} \hookrightarrow\PP^{2m+1}\, $$ yields
an embedding of $S$ as a smooth scroll of dimension $n+1$ and
degree $(n+1)d$ in $\PP^{2m+1}$. A general linear projection of
$S$ to $\PP^{n+2}$ gives a hypersurface scroll $\Sigma\subseteq
\PP^{n+2}$ of degree $(n+1)d$.

Consider, for instance, a surface $E_d$ in $\PP^3$ of degree $d$,
which we suppose to be very general. The above construction gives
$$S_d:=E_d\times\PP^1\hookrightarrow\Seg_{3,1}\hookrightarrow\PP^7\,,$$
and $S_d$ is a threefold of degree $3d$ in $\PP^7$. A general
linear projection of $S_d$ to $\PP^4$ yields a threefold scroll
$\Sigma_d$ of degree $3d$ in $\PP^4$. It is swept out by a
two-dimensional family of $(3d-3)$-secant lines to the double
surface $\Delta_\Sigma$ (see Lemma \ref{199}).

 The following version of the Albanese inequality follows
immediately from the Semistable Reduction Theorem \cite[\S
1]{Morr} and the Geometric Genus Criterion (see formula (1) on p.\
119 in \cite[\S 6]{Morr} or, in the surface case, formula (8) in
\cite[Ch.\ 5, \S 5]{KK}).

\blem\label{lem:limit} Let $X$ be a flat limit of a one-parameter
family of smooth, irreducible, projective varieties
of geometric genus $\rho$. Let $X_i$ be irreducible components of
$X_0$ with geometric genera $\rho_i$, $i=1,\ldots,h$. Then
\[ \rho\ge \sum_{i=1}^ h \rho_i\,.\]
\elem

 Recall that, in the notation as in (\ref{400}), the geometric
genus $\rho(E_d)$ of a smooth surface $E_d$ of degree $d$ in
$\PP^3$ is equal to $\rho(E_d)={{d-1}\choose{3}}=N_{d-4}+1$ (see e.g.\
\cite[Ch.\ 4, (5.12.2)]{KK}). The following lower bound on the
geometric genus of the double surface is an analog of (\ref{002})
in the case of surface scrolls.

\blem\label{lem:limit2} Let $\Sigma_d\subseteq \PP^ 4$ be a
threefold scroll of degree $3d$, constructed as before over a very
general surface $E_d$ in $\PP^3$ of degree $d\ge 5$ as a base.
Then for  the geometric genus $\rho_d$ of the double surface
$\Delta_{\Sigma_d}$ we have a lower bound
\begin{equation}\label{eq:pd}
\quad \forall d\ge 5\,.
\end{equation}
\elem

\bproof Degenerate $E_d$ to $E_{d-1}\cup E_1$, where $E_{d-1}$ and
$E_1$ are general. Then $S_d$ degenerates to the union of
$S_{d-1}$ and $S_1=\Seg_{1,2}$, meeting along the Segre image $X$
of $C\times \PP^ 1$, where $C=E_1\cap E_{d-1}$. Accordingly,
$\Sigma_d$ degenerates  in $\PP^4$ to the union of $\Sigma_{d-1}$
and $\Sigma_1$, the latter being a hypersurface of degree 3 with a
double plane. These threefolds intersect along the general
projection $Y$ of $X$, plus another surface $Z$. The limit of the
double locus $\Delta_{\Sigma_d}$ consists of the union of
$\Delta_{\Sigma_{d-1}}$, of the plane $\Delta_{\Sigma_{1}}$, and
of $Z$. The ruling determines a dominant rational map $Z\dasharrow
E_{d-1}$. So there is at least one component $Z'$ of $Z$ with
geometric genus $\rho'\ge\rho(E_{d-1})$.
 Now the first inequality in \eqref {eq:pd} follows from Lemma
\ref {lem:limit}. In particular, $\rho_5\ge\rho'\ge 1$. By
induction for every $d\ge 5$ we obtain $$\rho_d\ge\rho(E_{d-1})+
\rho_{d-1}\ge\sum_{k=5}^{d-1}
\rho(E_{k})+\rho_5\ge\sum_{k=0}^{d-2}
{{k}\choose{3}}={{d-1}\choose{4}}\,,$$ as required. \eproof

It would be interesting to find the precise value of $\rho_d$.

\begin{thm}\label{thm:pg}  Any irreducible surface contained in a
very general hypersurface of degree $3d\ge 15$ in $\PP^4$
 has geometric genus $\rho\ge \min\{\rho_d, N_{d-4}+1\}$. In
 particular,
 $\rho\ge \rho(E_{d})$  if $d\ge 8$.
\end{thm}

\bproof The argument is similar to that in the proof of
Proposition \ref {201}, so we will be brief. We let $X_0$ be the
scroll $\Sigma_d$ and  $X$ be a general hypersurface  in $\PP^4$
of degree $3d$. The pencil generated by $X_0$ and $X$ gives rise
as  usual to a flat family $f: \mathcal X\to \mathbb D$. Suppose
that the general fibre of this family contains an irreducible
surface $Y$ of geometric genus $\rho< \min\{\rho_d, N_{d-4}+1\}$.
By Lemma \ref {lem:limit2}  the limit $Y_0$ of such a surface in
the central fibre does not contain $\Delta_{\Sigma_d}$. By Lemma
\ref {lem:limit} all of its components have geometric genus
$\rho'\le\rho< \min\{\rho_d, N_{d-4}+1\}\le N_{d-4}+1$. Hence they
cannot dominate $E_d$, which has geometric genus $N_{d-4}+1$. Thus
all components of $Y_0$ pull--back to $S_d$ to surfaces with zero
intersection with the ruling. This yields a contradiction as in
the proof of Proposition \ref {201}.
\eproof

\brem\label{rem:bound2} G. Xu gave in \cite[Theorem 2]{Xu1} a
sharp lower bound for the geometric genus of an irreducible
divisor on a very general hypersurface of degree $d\ge n+2$ in
$\PP^ n$, with $n\ge 4$. Of course Theorem \ref {thm:pg} above is
weaker than Xu's result. However, the method of proof is simple
and it may possibly have further applications. Hence it would be
interesting to extend Theorem \ref {thm:pg} to other degrees
(non-divisible by $3$), as well as to higher dimensions. We wonder
also whether  in higher dimensions an analog of Proposition
\ref{prop:bound} holds. For instance, one can suggest by analogy
that on a very general threefold in $\PP^4$ of degree $\ge 6$, the
divisors of geometric genera $\rho'<\rho_d$ form bounded families.
\erem

\section{Degeneration to scrolls and Kobayashi hyperbolicity}\label{sec:brody}
\subsection{Limiting Brody curves and Hurwitz Theorem}
Let $V$ be a subvariety of a
 hermitian complex manifold.
 A {\em Brody curve} in $V$ is a holomorphic map
$f:\C\to V$ satisfying
$$\sup_{z\in\C} ||df(z)||=||df(0)||=1\,.$$ By \emph{Brody's
reparametrization lemma} (\cite {Bro}), if $V$ is proper and
non--hyperbolic then it contains a Brody curve. Furthermore, from
any sequence of Brody curves in $V$ one can extract a subsequence
converging to a Brody curve, which is called a {\em limiting Brody
curve}.

Assume there is a proper dominant map $\pi:V\to C$ onto a smooth
projective curve $C$. If general fibres $D_c=\pi^{-1}(c)$ ($c\in
C$) are non--hyperbolic, i.e., contain Brody curves, then every
special fibre $D_0:=D_{c_0}$ is non-hyperbolic as well and
contains limiting Brody curves. The Hurwitz Theorem imposes
constrains on limiting Brody curve with respect to the
singularities of $D_{0}$ (cf.\ e.g., \cite[\S 1]{SZ1},
\cite[Theorem 2.1]{Za1}, and \cite[Lemma 1.2]{Za2}). Let
$\Delta_0=\br (D_{0})$ be the set of multi--branch points of
$D_{0}$ such that locally the branches of $D_{0}$ are
$\Q$--Cartier divisors on $V$, and let $\Delta$ be the Zariski
closure of ${\Delta_0}$. Consider a limit $f:\Omega \to D_{0}$ of
a sequence of holomorphic maps $f_n:\Omega \to
D_{c_n}$, with $c_n\in C\setminus \{c_0\}$ such that $c_n\to c_0$,
where $\Omega\subseteq\C$ is a connected domain.  Hurwitz' Theorem
says that, if $f(\Omega)\cap\Delta_0\neq\emptyset$, then
$f(\Omega)\subseteq\Delta$. In particular, if $\Delta$ is
hyperbolic then any limiting Brody curve in $D_{0}$ is contained
in $D_{0}\setminus \Delta_0$. Hence if both $\Delta$ and
$D_{0}\setminus \Delta_0$ are hyperbolic then all fibres $D_{c}$
($c\neq c_0$) close enough to $D_{0}$ are hyperbolic as well (cf.\
\cite{Za1}).

\subsection{A hyperbolicity criterion for hypersurfaces in
$\PP^n$}
Let $X_0$, $X_\infty$ be distinct hypersurfaces in $\PP^n$ of
degree $d$. Typically, $X_\infty$ will be a general surface of
degree $d$ meeting $\Sing(X_0)$ in points, where locally $X_0$ is
a union of two smooth branches intersecting transversally.
Consider the associated linear pencil $\{X_t\}_{t\in \PP^ 1}$.

Assume that for a general $t\in\PP^1$ the hypersurface $X_t$ is
non--hyperbolic. Then there exists a sequence of Brody curves
$\varphi_n:\C\to X_{t_n}$ (with respect to the Fubini--Study
metric on $\PP^n$), where $t_n\to 0$, converging to a limiting
(non--constant) Brody curve $\varphi_0:\C\to X_{0}$.

\bprop \label{pr1} In the above setting, let $B=
X_\infty\cap\overline{\br( X_0)}$. If $\overline{\br( X_0)}$
and $(X_0\setminus\br (X_0))\cup B$ are both hyperbolic, then
$X_{t}$, for  $t\neq 0$ close enough to $0$, is hyperbolic as
well.
\eprop

\bproof  By Hurwitz' Theorem and the hypotheses, the image of
$\varphi$ cannot be contained in $\overline{\br( X_0)}$, and it
can meet $\overline{\br( X_0)}$ only at $\left(\overline{\br(
X_0})\setminus \br( X_0)\right)\cup B$. But then it is contained
in $(X_0\setminus\br (X_0))\cup B$, a contradiction.  \eproof

\brem\label{010} Hurwitz' Theorem cannot be applied at points in
$\left(\overline{\br( X_0})\setminus \br( X_0)\right)\cup B$, e.g.
at a \emph{pinch point} of $X_0\subseteq \PP^3$, where $X_0$ is
locally analytically isomorphic to the surface $x^ 2=yz^ 2$ in
$\mathbb A^ 3=\mathbb A_\C^ 3$ at the origin, or at a base
point of the pencil situated on $\br( X_0)$.

Indeed, consider a linear pencil of surfaces given in an affine
chart $\mathbb A^3$ of $\PP^3$ as $X_t=\{x^2-y^2z=t\}$. The origin
${\bf 0}\in \mathbb A^3$ is a pinch point of $X_0$ and is not a
base point of the pencil. Consider also the family of entire
curves
$$\phi_\tau: \C\to X_{t},\qquad u\longmapsto
(u^2+\tau,u,u^2+2\tau),\quad\mbox{where}\quad t=\tau^2\in\C\,.$$
The limiting entire curve $\phi_0(\C)\subseteq X_0$ passes through
the pinch point ${\bf 0}\in X_0$ and is not contained in the
singular locus $\{x=y=0\}=\br( X_0)\cup\{\bf 0\}$ of $X_0$.
\erem

\bcor\label{mcor} In the same setting as before, consider the
normalization $\nu:\bar X_0\to X_0$. Suppose that $\overline{\br
(X_0)}$ is hyperbolic and there is a morphism $\pi:\bar X_0 \to E$
onto a hyperbolic variety $E$ such that for every $x\in E$
\be\label{vr} \pi^{-1}(x)\setminus \nu^{-1}({\br (X_0)}\setminus
(X_\infty\cap \br (X_0))\ee is hyperbolic. Then any hypersurface
$X_t\neq X_0$ for $t$ close enough to $0$ is hyperbolic.
Consequently, a very general hypersurface of degree $d$ in $\PP^n$
is algebraically hyperbolic.
\ecor

\bproof We keep the notation introduced before. Suppose that for
$t\in \PP^ 1$ general, $X_t$ is not hyperbolic. Let $\phi_0:\C\to
X_0$ be a (non--constant) limiting Brody curve. Since its image
cannot be contained in $\overline{\br (X_0)}$, there is a pullback
$\tilde\phi_0:\C\to \bar X_0$. Since $E$ is hyperbolic, the
composition $\pi\circ\tilde\phi_0:\C\to E$ is constant. Hence
$\tilde\phi_0(\C)$ is contained in a fibre $\pi^{-1}(x)$ over a
point $x\in E$. Furthermore, it does not meet $\nu^{-1}({\br
(X_0)}\setminus (X_\infty\cap \br (X_0))$. Indeed, otherwise
$\phi_0(\C)$ would meet $\br (X_0)\setminus (X_\infty\cap \br
(X_0))$ and, by Hurwitz' Theorem, it  would be contained in
$\overline {\br (X_0)}$,  which is impossible. Then
$\tilde\phi_0(\C)$  lies in \eqref{vr}, a contradiction.
\eproof

\subsection{Applying scrolls to Kobayashi hyperbolicity}\label{ssec:d=67}
\bprop\label{200} We keep the notation as in Subsection
\ref{ssec:general}. Let $\Sigma\subseteq\PP^n$ be a hypersurface
scroll with ordinary singularities satisfying conditions
(C1)-(C5). Suppose that: \bnum\item[(i)] the base $E$ of $\Sigma$
and its double locus $\Delta_\Sigma$ are both hyperbolic;
\item[(ii)] for  a general hypersurface $X$ in $\PP^n$ of degree $d=\deg(\Sigma)$,
every ruling $F$ of $\Sigma$ meets $\br(\Sigma)$ in at least three
distinct points off $X\cap F$. \enum Then a general hypersurface
of degree $d$  in $\PP^n$ is hyperbolic. \eprop

\bproof  The assertion follows by applying
 Corollary \ref{mcor} with $X_\infty=X$, $X_0=\Sigma$,
 and $\overline{\br(X_0)}=\Delta_\Sigma$.
 \eproof
Consider a general sextic scroll of genus 2 as introduced in
Example \ref {sextic} and a general septic scroll, also of genus
2, with ordinary singularities in $\PP^ 3$. The latter scroll
exists according to Theorem \ref {thm:scrolls} and Remark \ref
{rem:3space}, (ii).

\blem\label{100} For a general scroll $\Sigma\subseteq \PP^ 3$ of
genus 2 and degree either $d=6$ or $d=7$, the following hold:
\begin{itemize}
\item [(i)]~the projection $\pi:\Delta_S\to E$ has only
simple ramifications; in particular $\Delta_S$ meets
every ruling in at least three distinct points;
\item [(ii)]~ no pair of pinch points on $S$ sit on the same ruling;
\item [(iii)]~ the rulings passing through the pinch points on $S$ are not
tangent to $\Delta_S$.
\end{itemize}
\elem

\bproof  We first treat the case $d=6$.

The conditions (i)--(iii) are open in  $\mathcal H=\overline
{\mathcal H_{6,2}}$. So it suffices to show that there is a
surface in $\mathcal H$ satisfying these conditions. The reducible
surface $\Sigma_0$ in Example \ref {sextic} could be used for
this, once we know that the analogues of (i)--(iii) hold for a
general elliptic quintic scroll. This is in fact the case, but we
do not dwell on this here. We use instead a different degeneration
of a general sextic scroll of genus 2. We keep the notation
introduced in Example \ref {sextic}.

A smooth quadric $\tilde Q$ in $\PP^4$ can be viewed as a
hyperplane section of the Grassmanian $\Gr(1,3)$ under the
Pl\"ucker embedding of $\Gr(1,3)$ in $\PP^5$. There exists a curve
$E_0$ of degree $6$ and arithmetic genus $2$ on $\tilde Q$, which
consists of three conics $\Gamma_0, \Gamma_1, \Gamma_2$ such that
$\Gamma_1, \Gamma_2$ are disjoint and intersect both $\Gamma_0$
transversally at two points. Indeed, it is enough to take two
general hyperplanes $H_1,\, H_2$ in  $\PP^4$ meeting in a plane
$L_0$ and two other general planes $L_i\subseteq H_i$, $i=1,2$,
and let $\Gamma_i=L_i\cap\tilde Q$, $i=0,1,2$.

The surface $\Sigma_0\subseteq\PP^3$, which corresponds to the
curve $E_0$, is the union of the three quadrics $Q_0, Q_1, Q_2$,
corresponding to $\Gamma_0, \Gamma_1, \Gamma_2$, respectively. We
may suppose that these quadrics are smooth. The surface $\Sigma_0$
belongs to  $\mathcal H$.
One has $Q_0\cap Q_i= F_{ij}\cup G_{ij}$, $i,j=1,2$, where the
lines $F_{ij}$'s correspond to the intersection points of
$\Gamma_0$ with $\Gamma_i$ and belong to the same rulings of $Q_0$
and $Q_i$, and $G_{ij}$ are lines of the other rulings of $Q_0$
and $Q_i$. Furthermore $Q_1\cap Q_2=\varrho$ is a smooth
quartic curve of genus $1$. By taking $Q_0, Q_1, Q_2$
sufficiently general, we may suppose that the lines $F_{ij},
G_{ij}$ are general in their rulings and $\varrho$ is also
general. We denote by $p_{ij;hk}$ the intersection of $F_{ij}$
with $G_{hk}$, where $i,j,h,k\in\{1,2\}$. We note that $\varrho$
meets $Q_0$ at the eight points $p_{ij;3-i,h}$, with $i,j,
h\in\{1,2\}$.

Regard now $\Sigma_0$ as a limit of a general sextic scroll
$\Sigma$ of genus 2. The points of $\Gamma_0\cap\Gamma_i$,
$i=1,2$, are smoothed when deforming $E_0$ to $E$, hence also the
lines $F_{ij}$ are. Therefore the limit of the smooth double curve
$C=\Delta_\Sigma$ is the curve
\[ C_0=\Delta_{\Sigma_0}=\varrho \cup \bigcup _{i,j=1,2} G_{ij} \]
of degree 8 and arithmetic genus 5. The limit on $\Sigma_0$ of the
ruling on $\Sigma$ is the union of the rulings of $Q_0, Q_1, Q_2$
containing the lines $F_{ij}$, $1\le i\le j$. By (\ref{f400})
there are 16 pinch points on $\Sigma$.  Similarly as in Example
\ref {quartic}, each of the eight points $p_{ij;ih}$, $i,j,h=1,2$
(not lying on $\varrho$) is the limit of two pinch points of
$\Sigma$. We call them \emph{limit pinch points}.

 The smooth normalization $S$ of $\Sigma$ specializes to
a partial normalization
$S_0$ of $\Sigma_0$, ruled over the same nodal base curve $E_0$.
The singular surface $S_0$ consists of three irreducible quadric surfaces
$\tilde Q_0,  \tilde Q_1,\tilde Q_2$ glued together along
the common rulings $\tilde F_{ij}$
in the same way as before.

The limit $\tilde C_0=\Delta_{S_0}$ of $\tilde C=\Delta_S$
is a nodal curve of arithmetic genus $17$. It maps to $E_0$
with degree $4$, and
consists of ten components:
\begin{itemize}
\item  two  copies $\varrho_i\subseteq \tilde Q_i$ of $\varrho$, each is
mapped with degree  $2$ to $\Gamma_i$, $i=1,2$;
\item two copies  $G_{i;hk}\subseteq \tilde Q_i$ of $G_{hk}$,
with $h,k=1,2$ and $i\in \{0,h\}$, eight curves in total. The
curves $G_{0;hk}$ and $G_{h;hk}$ are glued at two points  $p_{h1;hk}$ and
$p_{h2;hk}$.  Each of them is also glued to $\varrho_h$ at two
points, $h=1,2$. Hence the curves $\varrho_1$ and $\varrho_2$ meet in
the eight points $p_{ij;3-i,h}$, with $i,j,h=1,2$.
The four disjoint curves $G_{0;hk}$ on $\tilde Q_0$
are all mapped isomorphically
to $\Gamma_0$, whereas for  $h=1,2$ the two disjoint curves
$G_{h;hk}$ on $\tilde Q_h$ ($k=1,2$) are mapped
isomorphically to $\Gamma_h$.
\end{itemize}
Therefore, the limit of the $24$ ramification points of the projection
$\pi: \tilde C\to E$ are:
\begin{itemize}
\item [(a)] the  ramification points of the degree $2$ covers
$\varrho_i\to \Gamma_i$, $i=1,2$, in total $8$ such points;
\item [(b)] the connecting nodes of $\varrho_h$ with $G_{h;hk}$,
$k,h=1,2$, in total $8$ distinct such points, each counted with
multiplicity two.
\end{itemize}
We call these the \emph{limit ramification points}.

Part (i) follows from this description, our generality assumption,
and the observation that
every limit ramification point of type (b) smooths to two
ramification points on $\tilde C$ lying on different rulings.

As for (ii), the ruling $F_{ij}$  through $p_{ij;ih}$ misses all
limit pinch points other than $p_{ij;i,3-h}$. Consider a partial
deformation of $\Sigma_0$ to the union of a general elliptic
quartic scroll $\Sigma_0'$ and a quadric $Q'_1$ containing two
general rulings. This corresponds to a partial smoothing of $E_0$
to the union of an elliptic quartic curve $E'$, obtained by
smoothing $\Gamma_0+\Gamma_2$,  plus a conic $\Gamma'_1$
(specializing to $\Gamma_1$) meeting $E'$ transversally at two
points. In this way $\Sigma'_0$ has two double lines $R_1,R_2$
which respectively specialize to $G_{21}$ and $G_{22}$.  For a
fixed index $i\in \{0,1\}$, the two limit pinch points
$p_{2j;2i}$, $j=1,2$, deform to four pinch points of $\Sigma'_0$
on $R_i$, and, as we saw in Example \ref {quartic}, they are
general points on $R_1,R_2$ and are never pairwise on a ruling.

For the proof of (iii) note that, by generality assumptions, the
rulings through the limit ramification points of type (a) do not
contain any of the limit pinch points. In contrast, the rulings
through limit ramification points of type (b) do contain limit
pinch points. However, the same proof as for (ii) and generality
assumptions imply that, in a general deformation of $\Sigma_0$ to
$\Sigma$, this is no longer the case.

The case $d=7$ is similar, hence we will be as  brief as possible.
The closure of $\mathcal H_{7,2}$ contains points corresponding to
a surface $\Sigma_0=\Sigma'\cup P$, where $\Sigma'_0$ is a general
sextic scroll of genus 2 and $P$ is a general plane containing a
general ruling $F$. The intersection of $P$ with $\Sigma'$
consists of $F$ plus a  plane quintic curve $D$ of
genus $2$, which has four nodes $p_i$, $i=1,\ldots,4$. The
intersection of $C'=\Delta_ {\Sigma'}$ with $P$ consists of the
points $p_i$, $i=1,\ldots, 4$, and four more points $q_i\in F$,
$i=1,\ldots, 4$. The intersection of $D$ with $F$ consists
of the points $q_i$, $i=1,\ldots, 4$, and of a further point $q$
which is smooth on $\Sigma'$, so that $P$ is tangent to $\Sigma'$
at $q$.

The surface $\Sigma_0$ is the limit of a general scroll $\Sigma$
of degree $7$ and genus $2$. If $E$ is the base of $\Sigma$
regarded as a curve in $\Gr(1,3)$, this corresponds to $E$
degenerating to $E_0$, which is the union of a general sextic $E'$
of genus $2$ and a line $L$ meeting $E'$ transversally at
one points $f$, which corresponds to $F$. The limit of the ruling
of $\Sigma$ is the ruling of $\Sigma'$ plus the pencil in $P$,
corresponding to $L$, with center a general point of $F$.

The limit of $C=\Delta_\Sigma$ is the curve $C_0=C'\cup D$
of degree $13$. The points $p_i$, $i=1,\ldots, 4$, are limits of
the four triple points of $C$.
 The geometric genus of a partial smoothing of $C_0$ at the points $q_i$,
 $i=1,\ldots, 4$, is $10$.
 All this agrees with \eqref {f200},  \eqref {f300},
and \eqref {f400}.

The usual analysis shows that the limit of the $18$ pinch points
of $\Sigma$ (see \eqref {f500}) are the $16$ pinch points of
$\Sigma'$ plus the point $q$ counted with multiplicity $2$.

 The limit $\tilde C_0$ of $\tilde C=\Delta_S$ maps with degree five
 to the curve $E_0=E'\cup L$. It consists of:
 \begin{itemize}
\item  a copy $\tilde C'$ of $\Delta_{S'}$ which maps to $E'$ with
degree four;
\item a copy of the normalization $\bar D$ of $D$, which
maps isomorphically to $E'$ and meets $\tilde C'$
transversally at four points;
\item a copy of $D$
which maps  to $L$ with multiplicity five
via the projection induced by the ruling on $P$, and meets
$\tilde C'$ transversally at four points.
\end{itemize}
Hence the limit of the $44$ ramification points of the projection
$\pi: \tilde C\to E$ are
\begin{itemize}
\item the $24$ ramification points of the map $\tilde C' \to E'$;
\item the $12$ ramification points of the map $D \to L$;
\item  the $4$ connecting nodes of  $\tilde C'$ with $\bar D$,
each counted with multiplicity two.
\end{itemize}

With this in mind the proof proceeds similarly to the case $d=6$.
The details can be left to the reader. \eproof

\bthm\label{mthm} For every $d\ge 6$ there exists a hyperbolic
surface  in $\PP^3$ of degree $d$. Consequently, a very general
surface in $\PP^3$ of degree $d\ge 6$ is algebraically hyperbolic.
\ethm

\bproof For $d=6,7$ this follows from Corollary \ref{mcor} and
Lemma \ref{100}. For $d\ge 8$ one can consider e.g. a general
deformation of the union of two general cones in $\PP^3$ of
degrees $d_1,d_2$, where $d_1+d_2=d$ and $d_i\ge 4$ (see \cite {SZ2}).
\eproof

\brem\label{sz} Consider the union $X_0=X_1\cup X_2$ of projective
cones with distinct vertices in $\PP^4$ over two smooth hyperbolic
surfaces in $\PP^3$. According to \cite{SZ2}, $X_0$ can be
deformed to a smooth hyperbolic threefold in $\PP^4$ of degree
$\deg (X_1)+\deg (X_2)$. Thus there exist hyperbolic threefolds in
$\PP^4$ of any given degree $d\ge 12$. Consequently, a very
general threefold in $\PP^4$ of degree $d\ge 12$ is algebraically
hyperbolic. \erem

\providecommand{\bysame}{\leavevmode\hboxto3em{\hrulefill}\thinspace}


\begin{thebibliography}{KaMi}

\bibitem {Alb} G.~Albanese. Sulle condizioni perch\'e una curva algebrica riducibile
si possa considerare come limite di una curva algebrica
irriducibile. Rend.\ del Circolo Mat.\ di Palermo, {\bf 52}
(1928), 105--150.

\bibitem {ac2}  E. ~Arbarello, M. ~Cornalba.
Su di una proprieta' notevole dei morfismi
di una curva a moduli generali in uno spazio proiettivo.
 Rend.\ Sem.\ Mat.\ Univ.\
Politec.\ Torino {\bf 38} (1980), 2, 87--99.

\bibitem {ac1}  E.~Arbarello, M. ~Cornalba.
Footnotes to a paper of Beniamino Segre.
The number of $g^{1}_{d}$'s on a general $d$--gonal curve, and the
unirationality of the Hurwitz spaces of $4$--gonal and $5$--gonal
curves. Math.\ Ann.\ {\bf 256} (1981), 3, 341--362.

\bibitem{APS} E.~Arrondo, M.~Pedreira. I.~Sols.\
On regular and stable ruled surfaces in $\PP^3$.\ Algebraic curves
and projective geometry (Trento, 1988), 1--15. With an appendix of
R.~Hernandez, 16--18, Lecture Notes in Math., {\bf 1389},
Springer-Verlag, Berlin, 1989.

\bibitem{BCGMB} M.C.~Beltrametti, E.~Carletti,
D.~Gallarati, G.~Monti Bragadin.
Lectures on curves, surfaces and projective varieties. A classical
view of algebraic geometry. EMS Textbooks in Mathematics. European
Mathematical Society (EMS), Z\"urich, 2009.

\bibitem {MBer} M.-A.~Bertin.\ On the singularities of
the trisecant surface to a space curve. Le Matematiche {\bf 53}
(1998), Supplemento, 15--22.

\bibitem{Bo} F.~ Bogomolov. Families of curves on a surface of general type.
Soviet Math.\ Dokl.\ {\bf 18} (1977), 1294--1297.

\bibitem {Bo} P.~ Bonnesen.\ Sur les s\'eries lin\'eaires triplement
infinies de courbes alg\'ebriques sur une surface alg\'ebrique.
Bull. Acad. Royal des Sciences et de lettres de Danemarque {\bf 4}
(1906).

\bibitem {Bro} R.~Brody. Compact manifolds and hyperbnolicity.
Trans.\ Amer.\ Math.\ Soc., {\bf 235} (1978), 213--219.

\bibitem{CCFMLincei} A.~Calabri, C.~Ciliberto, F.~Flamini, R.~Miranda.
Degenerations of scrolls to unions of planes, {Rend.\ Lincei Mat.\
Appl.}, {\bf 17} (2006), 95--123.

\bibitem {chen} Xi~ Chen. Rational curves on K3 surfaces.
J.\ Algebraic Geom., {\bf 8} (1999), 245--278.

\bibitem{CLR}
L.~ Chiantini, A.-F.~ Lopez, Z.~ Ran. Subvarieties of generic
hypersurfaces in any variety. Math.\ Proc.\ Cambr.\ Philos.\ Soc.\
{\bf 130} (2001), 259--268.

\bibitem {CZ} C.~Ciliberto, M.~Zaidenberg.\
3-fold symmetric products of curves as hyperbolic hypersurfaces in
${\PP}^4$. Intern.\ J.\ Math., {\bf 14} (2003), 413--436.

\bibitem{ClR} H.~Clemens, Z.~Ran.\ Twisted genus bounds
for subvarieties of generic hypersurfaces. Amer.\ J.\ Math.\ {\bf
126} (2004), 89--120; erratum {\em ibid.} 127 (2005), 241--242.

\bibitem{DEG} J.P. ~Demailly, J. ~El Goul.\
Hyperbolicity of generic surfaces of high degree in projective
3-space. Amer.\ J.\ Math.\ {\bf 122} (2000), 515--546.

\bibitem{Do} I.V.~Dolgachev.\
Topics in classical algebraic geometry, Part I. Available at the site:\\
http://www.math.lsa.umich.edu/~idolga/topics1.pdf

\bibitem{Du} J. ~Duval.\ Une sextique hyperbolique dans ${\PP}^3(\C)$. Math.\
Ann.\ {\bf 330} (2004), 473--476.

\bibitem{Ei1}
L. ~Ein.\ Subvarieties of generic complete intersections. Invent.\
Math.\ {\bf 94} (1988), 163--169.

\bibitem{Ei2}
L.~Ein.\ Subvarieties of generic complete intersections. II.
Math.\ Ann.\ {\bf 289}  (1991), 465--471.

\bibitem{En} F.~Enriques.\ Le Superficie Algebriche. Nicola Zanichelli,
Bologna, 1949.

\bibitem{Fr1} A. ~Franchetta. \  Sulla curva doppia della proiezione di una superficie
generale dell'$S_4$, da un punto generico su un $S_3$,  Reale
Accad. d'Italia, Rend. Cl. Sci. Fis., Mat, Nat., s. VII, {\bf 2}
(1940), 282--288
 (also in  A. Franchetta, ``Opere Scelte''
 (C.~Ciliberto and C.~Sbordone Eds.), Giannini, Napoli, 2006, 19--27).

\bibitem{Fr2} A. ~Franchetta. Sulla curva doppia della proiezione di una
superficie generale dell'$S_4$, da un punto generico su un $S_3$,
Rend. Accad. Lincei, s. VII, {\bf 2} (1947), 276--279 (also in  A.
Franchetta, ``Opere Scelte'', C. ~Ciliberto and C. ~Sbordone Eds.,
Giannini, Napoli, 2006, 79--82).

\bibitem{Fr} A. ~Franchetta. \ Sulla variet\`a doppia della proiezione generica
di una variet\`a algebrica non singolare, Accad.\ Sci.\ Lettere ed
Arti di Palermo,\  IV, {\bf 14} (1953--54), 5--12 (also in  A.
~Franchetta, ``Opere Scelte'', C. ~Ciliberto and C. ~Sbordone
Eds., Giannini, Napoli, 2006, 205--212).

\bibitem{Fu} W. ~Fulton.\ Intersection Theory.
Springer--Verlag, Berlin, 1998.

\bibitem {Hal} M. ~Halic. A remark about the rigidity
of curves on K3 surfaces.
Collect.\ Math., {\bf 61} (2010), 323--336.

\bibitem{Ha} G.M. ~Hana.\ Rational curves on a general heptic fourfold.
In: ``Curves and Codes from Projective Varieties'', Ph.D.\
thesis. University of Bergen, Norway (2006).

\bibitem {har} R. ~Hartshorne. \ Algebraic Geometry.
Springer--Verlag, Berlin, 1977.

\bibitem {Kn}
A.L.~Knutsen.\
Remarks on families of singular curves with hyperelliptic normalizations.
Indag.\ Math.\ (N.S.) {\bf 19} (2008), 217--238.

\bibitem {KK} Vik.S.~ Kulikov, P.F.~ Kurchanov.  Complex algebraic varieties:
periods of integrals and Hodge structures. Algebraic geometry,
III, 1?217, Encyclopaedia Math.\ Sci.\ {\bf 36} (1998), Springer,
Berlin.

\bibitem{LM} S.\ S.-Y.\ Lu, Y.\ Miyaoka.
Bounding curves in algebraic surfaces by genus and Chern numbers.
Math.\ Res.\ Lett.\ {\bf 2} (1995),  663--676.

\bibitem{MQ}
M.~McQuillan.\ Holomorphic curves on hyperplane sections of
3-folds. Geom.\ Funct.\ Anal.\ {\bf 9} (1999), 370--392.

\bibitem{MP} E. ~Mezzetti, D. ~Portelli. A tour through some classical
theorems on algebraic surfaces. An.\ Stiint.\ Univ.\ Ovidius
Constanta Ser.\ Mat.\ {\bf 5} (1997), 51--78.

\bibitem{Mo} B. ~Moishezon.\ Complex surfaces and connected sums of complex
projective planes. Lect.\ Notes in Mathem.\ 630, Springer-Verlag,
Berlin, 1977.

\bibitem{Morr} D.R. ~Morrison.\ The Clemens-Schmid exact sequence and
applications, in ``Topics in Trascendental Algebraic Geometry'',
{Ann.\ of Math.\ Studies}, {\bf 106} (1984), 101--119.

\bibitem {No}  A. ~Nobile. \ Genera of curves varying in a family.
Ann.\ Sci.\ \'Ecole Norm.\ Sup.\ (4) {\bf 20} (1987), 465--473.

\bibitem{Pac1} G.~Pacienza.\ Rational curves on
general projective hypersurfaces. J.\ Alg.\ Geom.\ {\bf 12}
(2003), 245--267.

\bibitem{Pac2} G.~Pacienza.\ Subvarieties of general type
on a general projective hypersurface. Trans.\ Amer.\ Math.\ Soc.\
{\bf 356} (2004), 2649--2661.

\bibitem{Pau}
M.~Paun.\ Vector fields on the total space of hypersurfaces in the
projective space and hyperbolicity. Math.\ Ann.\ {\bf 340} (2008),
875--892.

\bibitem{Pi} R.~Piene. Some formulas for a surface in ${\PP}^{3}$.
 Algebraic geometry (Proc.\ Sympos., Univ.\
 Tromso, Tromso, 1977), 196--235.
 Lecture Notes in Math., {\bf 687}, Springer, Berlin, 1978.

\bibitem{Se} E.~ Sernesi. Partial desingularizations of families
of nodal curves.
Appendix in: F.~ Flamini, A.L.~ Knutsen, G.~Pacienza. On families
of rational curves in the Hilbert square of a surface, 679--682.
Michigan Math.\ J.\ {\bf 58} (2009), 639--682.

\bibitem{Sev} F. Severi,  Intorno ai punti doppi impropri
di una superficie generale dello spazio a quattro dimensioni e ai
suoi punti tripli apparenti, Rend. Circ. Mat. Palermo {\bf 15}
(1901), 33--51.

\bibitem{SZ0} B. ~Shiffman, M. ~Zaidenberg.\
Two classes of hyperbolic surfaces in ${\PP}^3$. International J.\
Math.\ {\bf 11} (2000), 65--101.

 \bibitem{SZ1} B. ~Shiffman, M. ~Zaidenberg.\
Constructing low degree hyperbolic surfaces in ${\PP}^3$. Houston
J.\ Math.\ (the special issue for S.\ S.\ ~Chern) {\bf 28} (2002),
377--388.

\bibitem{SZ2} B. ~Shiffman, M. ~Zaidenberg.\
New examples of hyperbolic surfaces in ${\PP}^3$. Funct.\ Anal.\
Appl.\ {\bf 39} (2005), 76--79.

\bibitem{Sh} D.~Shin.\ Rational curves on general hypersurfaces of degree 7 in
$\Bbb P^5$. Osaka J.\ Math.\ {\bf 44} (2007),  1--10.

\bibitem{Vo1} C.~Voisin.\ On a
conjecture of Clemens on rational curves on hypersurfaces. J.\
Diff.\ Geom.\ {\bf 44} (1996), 200--214.

\bibitem{Vo2} C.~Voisin.
A correction on ``A conjecture of Clemens on rational curves on
hypersurfaces'', J.\ of Diff.\ Geom.\ {\bf 49} (1999), 601--611.

\bibitem{Wa2}
L.C.~Wang.\ A remark on divisors of Calabi-Yau hypersurfaces.
Asian J.\ Math.\ {\bf 4} (2000), 369--372.

\bibitem{Xu1} G.~Xu.\ Subvarieties of
general hypersurfaces in projective space. J.\ Differential Geom.\
{\bf 39(1)} (1994), 139--172.

\bibitem{Xu3} G.~Xu.\ Divisors on general
complete intersections in projective space. Trans.\ Amer.\ Math.\
Soc.\ {\bf 348} (1996), 2725--2736.

\bibitem{Za1} M.~Zaidenberg.\ Stability of hyperbolic
embeddedness and construction of examples. Math.\ USSR Sbornik
{\bf 63} (1989), 351--361.
%
\bibitem{Za2} M.~Zaidenberg.\
Hyperbolicity of general deformations. Funct.\ Anal.\ Appl.\ {\bf
43} (2009), 113--118.

\bibitem {Zak1} F. L. Zak, {Tangents and secants of algebraic
varieties}, Translations of Mathematical Monographs, {\bf 127},
Amer. Math. Soc. 1993.

\bibitem {za} O.~Zariski.\  Dimension--theoretic characterization
of maximal irreducible algebraic systems of plane nodal curves of
a given order $n$ and with a given number $d$ of nodes. Amer.\ J.\
Math.\ {\bf 104} (1982), 209--226.

\end{thebibliography}
\end{document}